\documentclass[11pt,reqno]{amsart}

\usepackage{mathrsfs}
\usepackage{amsfonts}
\usepackage{amssymb}
\usepackage{amsthm}
\usepackage{amsmath}
\usepackage{latexsym}
\usepackage{tikz}

\newtheorem{theorem}{Theorem}[section]
\newtheorem{proposition}[theorem]{Proposition}
\newtheorem{lemma}[theorem]{Lemma}

\newtheorem{remark}[theorem]{Remark}

\newtheorem{assumption}[theorem]{Assumption}
\newtheorem{example}[theorem]{Example}

\makeatletter \@addtoreset{equation}{section} \makeatother

\newcommand{\pp}{\partial_+}
\newcommand{\pn}{\partial_-}
\def\tilde{\widetilde}

\newcommand{\ga}{\gamma}
 
\newcommand{\beq}{\begin{equation}}
\newcommand{\eeq}{\end{equation}}
\newcommand{\eps}{\varepsilon}
\def\TS{\textstyle}
\def\com#1{\quad{\textrm{#1}}\quad}
\def\eq#1{(\ref{#1})}
\def\nn{\nonumber}
\def\ol{\overline}

\def\dbyd#1{\TS{\frac{\partial}{\partial#1}}}

\begin{document}

\title{Singularity formation for the compressible Euler equations}

\author[G. Chen]{Geng Chen}\address{Department of Mathematics,
University of Kansas, Lawrence, KS 66044 USA}
\email{\tt gengchen@ku.edu}
\author[R. Pan]{Ronghua Pan}\address{School of Mathematics,
Georgia Institute of Technology, Atlanta, GA 30332 USA}
\email{\tt panrh@math.gatech.edu}
\author[S. Zhu]{Shengguo Zhu}\address{ Department of Mathematics, Shanghai Jiao Tong University, Shanghai 200240, P.R.China}
\email{\tt zhushengguo@sjtu.edu.cn}

\date{March 31, 2015}


\begin{abstract}
It is well-known that singularity will develop in finite time for
hyperbolic conservation laws from initial nonlinear compression no matter how
small and smooth the data are. Classical results, including 
Lax \cite{lax2}, John \cite{john},
 Liu \cite{Liu1}, Li-Zhou-Kong \cite{lizhoukong0}, confirm that when initial data are
small smooth perturbations near
constant states, blowup in gradient of solutions occurs in finite time if initial data contain 
any compression in some truly nonlinear characteristic field, under some structural conditions.
A natural question is that: Will this picture keep true for
large data problem of physical systems such as compressible Euler equations?
One of the key issues is how to find an
effective way to obtain sharp enough control on density lower bound, which is known to decay to zero as time goes to infinity for certain class of solutions. In this paper, we offer a simple way to characterize the decay of density lower bound in time, and therefore successfully classify the questions on singularity formation in compressible Euler equations.
For isentropic flow, we offer a
complete picture on the finite time singularity formation from smooth initial
data away from vacuum, which is consistent with the small data theory.
   For adiabatic flow, we show a striking observation that
initial weak compressions do not necessarily develop singularity in finite time. Furthermore, we follow \cite{G6} to introduce the critical 
strength of nonlinear compression, and prove that if the compression is 
stronger than this critical value, then singularity develops in finite time, 
and otherwise there are a class of initial data admitting global smooth solutions 
with maximum strength of compression equals to this critical value. 
\end{abstract}

\maketitle

\section{Introduction}

The compressible Euler equations are the oldest system of nonlinear PDEs modeling the motion of gases.
Under Lagrangian coordinates, the compressible Euler equations in one space dimension take 
the following form

\begin{align}
\tau_t-u_x&=0\,,\label{lagrangian1}\\
u_t+p_x&=0\,,\label{lagrangian2}\\
\Big(\frac{1}{2}u^2+ \mathcal{E} \Big)_t+(u\,p)_x&=0\,, \label{lagrangian3}
\end{align}
where $x$ is the Lagrangian spatial variable, $t\in\mathbb{R}^+$ is the time. $\tau=\rho^{-1}$ denotes the specific volume for the density $\rho$.
$p$, $u$ and $\mathcal{E}$ stand for the pressure, the velocity, and the  specific  internal energy, respectively. For polytropic ideal gases, it holds that 
\beq
   p=K\,e^{\frac{S}{c_v}}\,\tau^{-\gamma},\ \  \mathcal{E}=\frac{p\tau}{\gamma-1}\, , \ \ \gamma>1\,,
\label{introduction 3}
\eeq
where $S$ is the entropy, $K$ and $c_v$ are positive constants,
see~\cite{courant} or \cite{smoller}.  For most gases, the adiabatic exponent $\gamma$ lies between $1$ and $3$, that is $1<\gamma<3$.

For $C^1$ solutions,
it follows that (\ref{lagrangian3}) is equivalent to the ``entropy
equation'':
\beq
   S_t=0\,.
\label{s con}
\eeq
Therefore, when entropy is constant, the flow is called isentropic, then
(\ref{lagrangian1}) and (\ref{lagrangian2}) become a closed system,
known as the $p$-system (or isentropic Euler equations)
\begin{align}
\tau_t-u_x&=0\,,\label{p1}\\
u_t+p_x&=0\,,\label{p2}
\end{align}
with
\beq\label{p3}
   p=K_1\,\tau^{-\gamma}\,,
\eeq
where,  $K_1>0$ is a constant.

Compressible Euler equations (1.1)--(1.3)  and p-system (1.6)--(1.7)
are two of the most important physical models for hyperbolic conservations laws
\beq\label{CL}
{\bf u}_t+{\bf f}({\bf u})_x=0\,,
\eeq
where
${\bf u}={\bf u}(x,t)\in
\mathbb{R}^n$ is the unknown vector and  ${\bf f}:\mathbb{R}^n\rightarrow\mathbb{R}^n$
is the nonlinear flux. 
It is a general belief that system (\ref{CL})
typically develops discontinuity singularity, i.e. shock wave, no matter how small and smooth the initial data are. This belief has been justified in a series of beautiful works by Lax \cite{lax2} in 1964 for general systems with two unknowns, and by 
\cite{john, lizhoukong0,lizhoukong,Liu1} for general $n\times n$ systems. These results confirm that for general strictly hyperbolic systems, if the initial datum is a generic small smooth perturbation near a constant equilibrium, then the initial compression (negative spatial derivatives of gradient variables) in any truly nonlinear (not weakly linearly degenerate \cite{lizhoukong0}) characteristic field develops singularity in finite time. Such lack of
regularity is the major difficulty in analyzing these systems. With enormous efforts, the well-posedness theories of
small total variation
solutions for \eqref{CL} including compressible Euler equations and p-system are fairly well understood \cite{Bressan,Dafermos, Glimm}. The next natural question is on the theory of large data, which is, however, widely open. Even for 
some important physical systems, such as compressible Euler equations and p-system, the basic question, like if singularity will form in finite time, is not completely understood when the smallness condition on the initial data is missing. We will address this open problem in this paper for p-system and full compressible Euler equations.  

The beautiful result of Lax \cite{lax2} along with some expositions such as \cite{Evans} left readers an impression that, at least for p-system, for $C^1$ initial data away from vacuum, singularity will form 
in finite time if and only if there is some compression (negative spatial derivatives of gradient variables, we refer the readers to Remark 2.4 below for the definition) initially, without smallness assumption. This, however, 
is not quite accurate. When adopting \cite{lax2} to p-system, the control on a crucial term $\frac{1}{c(v)}$ for the sound speed $c(v)=\sqrt{-p_v}$  is very important. This term is singular if density tends to zero. On the other hand,
the $L^\infty$ estimate through Riemann invariants offers an upper bound of density, without control on the lower bound. In the case of small solutions, one can actually choose the smallness of the perturbation carefully, so that the perturbation remains small comparing to the positive lower bound of initial density uniformly in time. Such choice of smallness gives a positive constant lower bound (say one half of the lower bound of  initial density) for density.  However, when initial data are large, this becomes a serious issue. In general, it is not possible to have a positive constant lower bound for density. Indeed,  a Riemann problem connecting two extreme sides of
two interacting strong rarefaction waves generates vacuum instantaneously when $t>0$,  \cite{smoller}.
Smoothing out this data implies the existence of a $C^1$-solution such that $\inf_{x}\rho(x,t)\rightarrow0$
as $t\rightarrow+\infty$. An example of Lipschitz continuous solutions can be found in the Section 82 in \cite{courant} using a method originally discussed in \cite{Riemann}, where the density decays in time at a rate of $\frac{O(1)}{1+t}$. If one looks into this problem 
more carefully, the argument in \cite{lax2} is valid only for p-system with large initial data and with pressure law (\ref{p3}) when $\gamma\ge 3$, which does not include the most practical case $1<\gamma<3$ in gas dynamics. In fact, when $\gamma\ge 3$, the control of lower bound of density is not needed, see  also a generalization to full Euler equations by Chen, Young and Zhang \cite{G6}. Therefore, the real matter of the open problem is to establish the finite time singularity formation for 
both p-system and full compressible Euler equations for the most physical case $1<\gamma<3$. A more 
in-depth discussion on Lax's result \cite{lax2} will be presented in section 2 of this paper. 

The main purpose of this paper is to establish the finite time singularity formation result for both p-system and full Euler equations without the smallness assumption on initial data, when gases are in physical regime, i.e., 
$1<\gamma<3$. We introduce a brand new elementary and neat approach to establish the time-dependent density lower bound, which is good enough to achieve our characteristic analysis leading to the finite time 
singularity formation results even when initial data are large. 

For isentropic flow with $\gamma$-law pressure, our result shows that if the initial datum is smooth with a positive lower bound for density, then the classical solution of the Cauchy problem of p-system breaks down in finite time {\bf if and only if} there is an initial compression. The precise statement is in Theorem \ref{p_sing_thm}, and the definitions of rarefaction and compression are given in Remark \ref{rem2.4}. We emphasize that the approach introduced in the proof of this theorem is neat and elementary, but also very powerful. The key new estimates are given by Lemmas \ref{lemma_p_2} and \ref{density_low_bound_1-3}. Note the time dependent lower bound we proved for density is not in the optimal order, but it is good enough  for the singularity formation problem. Furthermore, this approach is applicable to the full 
Euler equations for non-isentropic flows. 

We now make a brief remark on the time-dependent lower bound for density. In many literatures, 
using Eulerian formulation, and mass equation, the following estimate of density 
$$\displaystyle \inf_x \rho_0(x) \exp\{-\int_0^t \|u_{x'}(\cdot, \sigma)\|_{L^\infty}\ d\sigma\} \le \rho(x, t), $$
has been obtained for $x'$ the Eulerian space variable, c.f. \cite{LiBook}. We note that, since there is a possibility that the blowup of gradient of $u$ and the vanishing of density may happen at the same time, it is very difficult to use this estimate in the argument of proving global regularity or singularity formation. When initial data are purely rarefactive (see Remark 2.4 below for definition), L. Lin \cite{lin2} proved that the density of  any Lipschitz solution of p-system has a 
positive lower bound of order $\frac{1}{1+t}$ through a relatively complicated approximation 
generated by a polygonal scheme. This result, however, does not apply to the case when initial data contain compression, which is the mechanism for singularity formation. One of the main contributions of this paper is to provide a good enough new time-dependent estimate on the density lower bound for generic $C^1$ initial data away from vacuum, when $1<\gamma<3$. The idea we developed here is simple and neat, but does not offer the optimal rate $\frac{1}{1+t}$, which is achieved through a much more complicated method in our preprint \cite{CPZ} for generic $C^1$ initial data away from vacuum. 

From the discussion above, we see that for isentropic flow, the singularity formation theory is not different no matter the data is small or large. One may thus expect a similar picture for the non-isentropic flow. However, the life is very complicated for non-isentropic flow. When initial data are small, under a so-called nonlinear wave condition,
\cite{lizhoukong0,lizhoukong,Liu1} showed that if the initial datum is a generic small smooth perturbation near a constant equilibrium, then the initial compression in any truly nonlinear (not weakly linearly degenerate \cite{lizhoukong0}) characteristic field develops singularity in finite time, like in the p-system. When initial data are large, this expectation is not true for full Euler equations. 
In Section 3.5, we will provide an explicit example showing that for certain class of non-trivial initial data, which might be even periodic in space variable with non-zero derivatives, global classical solutions exist. This is in a sharp contrast to the isentropic case, where the non-trivial periodic initial data lead to finite time singularity formation of classical solutions.  We remark that this class of initial data do not satisfy the so-called nonlinear wave condition in the blowup results of \cite{Liu1}. Therefore, in order to prove finite time singularity formation results for full Euler equations with large data, it is natural to impose some conditions to exclude this class of initial data, see also some related discussion in \cite{G6}.  In Section 3, we will identify such kind of conditions and successfully establish the finite time singularity formation results when initial compression is merely stronger than a critical value, which can be attained by the global classical solutions constructed in our example. More detailed discussion will also be provided at the end of this paper.  For full Euler system, singularity formation results were proved in \cite{G3, G6} for $\gamma\ge 3$ when the density lower bound is not needed, and for $1<\gamma<3$ with the help of an {\it a priori assumption } on the density 
lower bound. For small solutions, similar to isentropic case, smallness conditions give a positive 
constant lower bound for density, say one half of initial density lower bound. For large solutions, sufficiently good time-dependent density lower bound estimate is needed for the proof of finite time singularity formation, when $1<\gamma<3$. Our idea introduced for isentropic case can be generalized to full Euler system, and gives a good enough estimate on density lower bound. The singularity formation for p-system with general pressure law is discussed in the appendix.
 
When some restrictions, such as compactness of support of the initial data near certain constant equilibrium, are imposed, there are some wonderful results on the finite time singularity formation for compressible Euler equations in higher space dimensions. We refer the readers to some of these results, see \cite{tms1,rammaha,sideris} for classical compressible Euler equations, and see \cite{Chris, ps} for relativistic
Euler equations. The results in this paper in one space dimension offer more complete and clear pictures on 
the mechanism, occurrence, and the type of singularity formations.

\section{Singularity formation for p-system}\label{section2}
In this section, we study singularity formation 
for p-system \eqref{p1}$\sim$\eqref{p3}. The proof of our main theorem (Theorem \ref{p_sing_thm}) 
is based on the study of Lax's characteristic decomposition established first for general hyperbolic system with two unknowns in \cite{lax2}. For the readers' convenience, in Subsection \ref{subsection2.1}, we first review this well-known result of Lax \cite{lax2}. 
Then in Subsection
\ref{subsection2.2}, we present a careful adoption of  Lax's method in \cite{lax2} to p-system with $\gamma$-law pressure. We will then explain why Lax's result \cite{lax2} for small smooth initial data for 
general $2\times 2$ system actually offers the singularity formation result for p-system without smallness 
restrictions on initial data for $\gamma$-law pressure provided that $\gamma\geq 3$. We will also spell out why his result does not include the most physical cases when $1<\gamma<3$. In the latter case, a careful study 
on the lower bound of density is needed, which is achieved in Subsection \ref{subsection2.3}, leading to 
the first main result of this paper Theorem \ref{p_sing_thm}. 

\subsection{Lax's result for $2\times 2$ systems\label{subsection2.1}}

This part is basically taken from Lax's paper \cite{lax2} in 1964. Consider a system of two first-order partial differential equations
\beq\label{22_1}
\begin{split}
u_t+f_x=0\,,\\
v_t+g_x=0\,,
\end{split}
\eeq
where $f$ and $g$ are functions of $u$ and $v$. Carrying out the differentiation in \eqref{22_1}, we obtain
\beq\label{22_2}
{\bf u}_t
+A\,
{\bf u}_x=0\,,
\eeq
where
\[{\bf u}=
\left(
\begin{array}{l}
u\\
v
\end{array}\right)
\qquad
\text{and}
\qquad
A=\left(
\begin{array}{lr}
f_u&f_v\\
g_u&g_v
\end{array}\right)\,.
\]
Suppose that this system is strictly hyperbolic, i.e. 
the matrix $A$ has real and distinct eigenvalues $\lambda<\mu$
for relevant values of $u$ and $v$. Use ${\bf l}_{\lambda}$ and ${\bf l}_{\mu}$
to denote the left eigenvectors of $A$ corresponding to  eigenvalues $\lambda$ and $\mu$,
respectively.

Multiplying \eqref{22_2} by ${\bf l}_{\lambda}$ and ${\bf l}_{\mu}$ respectively, we have 
\[
{\bf l}_{\lambda}\cdot {\bf u}'=0\,,\qquad {\bf l}_{\mu}\cdot {\bf u}^\backprime=0\,,
\]
where we denote
\[
\prime=\partial_t+\lambda\partial_x\,,\qquad \backprime=\partial_t+\mu\partial_x\,.
\]
Suppose there exist integrating factors $\phi$ and $\psi$  such that $\phi {\bf I}_{\lambda} =\nabla_{(u,v)} w(u, v)$, and $\psi {\bf I}_{\mu} =\nabla_{(u,v)} z(u, v)$, and therefore
\beq\label{22_3}
w'=\phi_\lambda\,{\bf l}_{\lambda}\cdot {\bf u}'=0,\,  \  z^\backprime=0.
\eeq
for some functions $w(u,v)$ and $z(u,v)$, which are called Riemann invariants
along characteristics with characteristic speeds $\lambda$ and $\mu$, respectively. Therefore, the $L^\infty$ norms on $w$ and $z$ are bounded by the initial data. We remark that 
such $\phi$ and $\psi$ always exist at least locally. Thus, for general hyperbolic systems with two unknowns, there always exist two Riemann invariants
for different families,
if we restrict the initial data to be a small perturbation near a constant equilibrium. This is also one 
of the reasons that Lax's result in \cite{lax2} is a small data theory. For many general systems, such 
as  p-system, Riemann invariants are naturally well-defined globally, therefore, the smallness restriction 
is not an issue for this step. 

We focus on $w$, the case on $z$ can be treated in a similar manner. Differentiating  $w'=0$  in \eqref{22_3} on $x$, we have 
\beq\label{22-4}
w_{tx}+\lambda w_{xx}+\lambda_{w} w_x^2+ \lambda_{z} w_x z_x=0\,.
\eeq
Also by  \eqref{22_3}, we observe from 
\[
0=z^{\backprime}=z'-(\lambda-\mu)z_x\,,
\]
that
\beq\label{22-5}
z_x=\frac{z'}{\lambda-\mu}\,.
\eeq
Substituting \eqref{22-5} into \eqref{22-4} and denoting 
\[\alpha:=w_x\,,\]
one finds
\beq\label{22-6}
\alpha'+\lambda_{w} \alpha^2+\frac{\lambda_{z}}{\lambda-\mu} z' \alpha=0\,.
\eeq

Let $h$ be a function of $w$ and $z$ satisfying
\[
h_z=\frac{\lambda_{z}}{\lambda-\mu} \,.
\]
Using $w'=0$ in  \eqref{22_3}, we have
\[
h'=h_w w'+h_z z'=\frac{\lambda_{z}}{\lambda-\mu} z'\,.
\]
This, together with (\ref{22-6}), gives
\beq\label{al_lax}
\alpha'+\lambda_{w} \alpha^2+ h' \alpha=0\,.
\eeq
Multiplying \eqref{al_lax} by $e^h$ and denoting
\[
\tilde{\alpha}:=e^h \alpha\,,
\]
we finally obtain
\beq\label{22_8}
\tilde{\alpha}'=-a {\tilde{\alpha}}^2\,
\eeq
with
\[
a:=e^{-h}\lambda_{w}\,.
\]

This Riccati type equation gives us a clear passage to study the singularity formation and/or
global existence of classical solutions for hyperbolic system with two unknowns. 
In fact, we could formally solve gradient variable $\tilde{\alpha}$ along a characteristic $x(t)$ defined by 
$$\frac{d x(t)}{dt}=\lambda,\ x(0)=x_0,$$
to obtain 

\[
\frac{1}{\tilde{\alpha}(x(t), t)}=\frac{1}{\tilde{\alpha}(x_0,0)}+\int_0^t a(x(\sigma), \sigma) ~d\sigma
\]
where the integral is taken along the characteristic  curve $x(t)$.

Note that $a\neq 0$ if $\lambda_w\neq 0$, which is corresponding to the nonlinearity of the system. One 
does not expect singularity formation for linearly degenerate fields \cite{Liu1}. For simplicity, suppose that $a$ is always non-zero, which is also satisfied by the solution of
p-system if initially $a\neq0$. To fix the idea, we only consider the case $a>0$. If 
$\tilde{\alpha}(0)<0$, i.e. initial solution is compressive somewhere in the $\lambda$ direction, then $\tilde{\alpha}(t)$ breaks down if there exists a time $t_*>0$ such that 
\beq\label{22_7}
\int_0^{t_*} a(x(\sigma), \sigma) ~d\sigma= -\frac{1}{\tilde{\alpha}(x_0,0)} \,.
\eeq
which could be relaxed to 
\beq\label{22_9}
\int_0^{\infty} a(x(\sigma), \sigma) ~d\sigma= \infty \,.
\eeq
In \cite{lax2}, Lax considered the hyperbolic system with uniformly 
strict hyperbolicity, i.e. characteristic speeds $\lambda$ and $\mu$ are uniformly away from each other.
With the help of smallness condition on initial data, there is a positive constant ${\bar a}$ such that $a\ge {\bar a}>0$, if the initial data are chosen so,  hence 
\eqref{22_9} is automatically justified. 

When smallness condition on the initial data is lacking, in principle, one expects the similar results 
following Lax \cite{lax2} if the Riemann invariants are defined globally, and \eqref{22_9} is satisfied. 
For p-system \eqref{p1}$\sim$\eqref{p3}, the Riemann invariants are defined globally, it remains to check 
\eqref{22_9}. We will explain how far Lax's theory can reach in the next subsection.

\subsection{Lax's large data theory on p-system: $\gamma\ge 3$.\label{subsection2.2}}

We adopt Lax's theory on singularity formation to the following Cauchy problem of p-system \eqref{p1}$\sim$\eqref{p3}, i.e., 
\begin{equation}\label{cp}
\begin{cases}
&\tau_t-u_x=0\,,\\
&u_t+p_x=0\,,   \ p=K_1\,\tau^{-\gamma}, \\
&\tau(x,0)=\tau_0(x),\ u(x, 0)=u_0(x),
\end{cases}
\end{equation}
where,  $K_1>0$ and $\gamma>1$ are constants. If the initial data are chosen to be a small smooth perturbation near a constant state $({\bar \tau}, {\bar u})$ with ${\bar \tau}>0$, then Lax's theory 
in \cite{lax2} applied directly. Our main concern in this subsection is how far it could reach when initial 
data are not small. 

From now on, we make the following assumption throughout the rest of Section \ref{section2}:
\begin{assumption}\label{p-assumption} Assume that $(\tau_0(x), u_0(x))$ are $C^1$ functions, and there are uniform positive constants $M_1$ and $M_2$ such that 
$$\|(\tau_0, u_0)(x)\|_{C^1}\le M_1,\ \tau_0\ge M_2.$$
\end{assumption}

A direct calculation shows that \eqref{cp} has two characteristic speeds 
$$\lambda=-\mu=-c,$$
where $c$ is the Lagrangian sound speed 
\beq
  c:=\sqrt{-p_\tau}=
  \sqrt{K_1\,\gamma}\,{\tau}^{-\frac{\gamma+1}{2}}\,.
\label{c def}
\eeq
The forward and backward characteristics are defined by
\[
  \frac{dx^{+}}{dt}=c \com{and} \frac{dx^{-}}{dt}=-c\,,
\]
respectively. We denote the corresponding directional derivatives along them by
\[
  \pp := \dbyd t+c\,\dbyd x \com{and}
  \pn := \dbyd t-c\,\dbyd x\,,
\]
respectively. 
Furthermore, introducing the following useful quantity, c.f.
\cite{G3}, 
\beq\label{3.00}
   \eta          := \int^\infty_\tau{c\,d\tau} =\TS\frac{2\sqrt{K_1\gamma}}{\gamma-1}\,
\tau^{-\frac{\gamma-1}{2}}>0\,,
\eeq
the globally defined Riemann invariants of \eqref{cp} are
\beq\label{3.0}
  r:=u-\,\eta \com{and} s:=u+\,\eta\,,
\eeq
which satisfy
\beq\label{3.1}
  \pp s=0 \com{and} \pn r=0\,,
\eeq
respectively.

Since $\eta$, $p$ and $c$ all are functions of $\tau$, their relations are as follows

\begin{equation}
  \tau=K_{\tau}\,\eta^{-\frac{2}{\gamma-1}}\,,\quad
  p=K_p\, \eta^{\frac{2\gamma}{\gamma-1}}\,,\ \quad
  c=\sqrt{-p_\tau}=K_c\,  \eta^{\frac{\gamma+1}{\gamma-1}}\,,
\end{equation} 
where $K_\tau$, $K_p$ and $K_c$ are positive
constants given by
\beq
  K_\tau:=\Big(\frac{2\sqrt{K_1\gamma}}{\gamma-1}\Big)^\frac{2}{\gamma-1}\,,
\quad
  K_p:=K_1\,K_\tau^{-\gamma},\com{and}
  K_c:=\TS\sqrt{K_1\gamma}\,K_\tau^{-\frac{\gamma+1}{2}}\,.
\label{Kdefs}
\eeq
Clearly, one has
\beq
   K_p=\TS\frac{\gamma-1}{2\gamma}K_c \com{and}
   K_\tau K_c=\frac{\gamma-1}{2}\,.
\label{KpKcRela}
\eeq
In this paper, we always use $K$ with some subscripts to denote
positive constants. We will not alert the readers again if there is no ambiguity.

We observe from \eqref{3.1} that the $L^\infty$ norm of $(r, s)$ are bounded by the initial data, which leads to a uniform $L^\infty$ bounds on $u$ and $\eta(\tau)$ with the help of \eqref{3.0}. From \eqref{3.00}, one 
finds the uniform positive lower bound on the specific volume $\tau$, or equivalently, the uniform upper bound on the density $\rho=\frac{1}{\tau}$. However, we remark that, such estimates do not offer any control 
on the lower bound of density $\rho$ (or, upper bound of $\tau$).

Following the procedure of last subsection in deriving \eqref{22_8}, c.f. \cite{G3}, the good gradient variables are 
\[
  y :=  \eta^{\frac{\gamma+1}{2(\gamma-1)}}\,s_x 
\com{and}
  q :=        \eta^{\frac{\gamma+1}{2(\gamma-1)}}\,r_x\,,
\]
which satisfy the following Riccati type equations:
\begin{align}
  \partial_+ y &= - a_2 \, y^2\,, \label{p_y_eq}\\
  \partial_- q &= - a_2 \, q^2\,, \label{p_q_eq}
  \end{align}
where
\begin{align}\label{a2}
  {a}_2 &:= K_c\,{\TS\frac{\gamma+1}{2(\gamma-1)}}\,
        \eta^{\frac{3-\gamma}{2(\gamma-1)}}\,.
        \end{align}

We note the behavior of ${a}_2$ is purely determined by $\eta$. Since $\eta$ has a uniform upper bound, 
when $\gamma\ge 3$, there exists a uniform constant ${\bar a}_2>0$, such that $a_2\ge {\bar a}_2$. 
In this case \eqref{22_9} is justified, and Lax's theory applies without smallness condition. 

{\begin{proposition}\label{prop_p_1}(A corollary from \cite{lax2})
Assume that $(\tau_0(x), u_0(x))$ satisfy the conditions in Assumption \ref{p-assumption}. When $\gamma\geq3$, classical solution of (\ref{cp}) 
breaks down if there is a point $x_*\in\mathbb{R}$ such that 
\beq\label{laxcompress}
 s_x(x_*,0)<0 \com{or}  r_x(x_*,0)< 0\,.
\eeq
\end{proposition}}
\begin{proof}
We will show that if  $s_x(x,0)<0$ or $r_x(x,0)< 0$ for some $x$, then singularity forms in finite time.  Without loss of generality, we assume that $s_x(x^*,0)<0$, then $y(x^*,0)<0$ for some $x^*$.  Denote the forward characteristic passing $(x^*,0)$ as $x^+(t)$. By  \eqref{p_y_eq},
\beq\label{p_prop_2_proof}
\frac{1}{y(x^+(t), t)}=\frac{1}{y(x^*, 0)}+\int_0^t \, {a_2}(x^+(\sigma), \sigma)\;d\sigma\,,
\eeq
where $a_2\ge {\bar a}_2$ for some uniform constant ${\bar a}_2>0$. Therefore, the 
right hand side of \eqref{p_prop_2_proof} approaches zero in finite time, which means singularity happens in finite time.
\end{proof}

However,  for most physical gases where $1<\gamma<3$, the positive lower bound of the function ${a}_2$
requires positive lower bound of density. We remark that, for generic smooth initial data without initial vacuum, even for the global smooth solutions of \eqref{cp}, the density does not have constant positive lower bound in general.  An example of Lipschitz continuous solutions can be found in \cite{Riemann} and 
some detailed discussion in Section 82 of \cite{courant}, where the density decays in time at a rate of $\frac{1}{1+t}$.  In view of (2.10), the lower bound of $a_2$ cannot decay too fast for possible singularity formation in finite time. Therefore, a new idea is 
needed to obtain a sufficiently sharp control on the time-dependent positive lower bound of density. This will be one of our main contributions in this paper, which will be addressed in the next subsection.  

\subsection{Singularity formation in p-system: $\gamma>1$\label{subsection2.3}}
In this section, for all $\gamma>1$, we prove the singularity formation for the Cauchy problem in p-system when initial data contain some compression, and otherwise the global existence of smooth solutions. This is achieved by establishing a sharp enough time-dependent positive lower bound on density.  The following theorem is the first main result of this paper. 

\begin{theorem}\label{p_sing_thm} For $\gamma>1$, if $(\tau_0(x), u_0(x))$ satisfy conditions in Assumption \ref{p-assumption}, then the Cauchy problem (\ref{cp}) has a unique 
global-in-time classical solution if and only if 
\beq\label{p_lemma_con}
 s_x(x,0)\geq0 \com{and}  r_x(x,0)\geq 0,\com{for all} x\in\mathbb{R}\,.
\eeq
\end{theorem}

\begin{remark}\label{rem2.4}
At  a point $(x,t)$,  the solution of \eqref{cp} is said to be forward rarefactive (resp. compressive) if   $s_x(x,t)\geq0$ (resp. $s_x(x,t)<0$); the solution is said to be backward rarefactive (resp. compressive) if   $r_x(x,t)\geq0$ (resp. $r_x(x,t)<0$).

Hence this theorem can be understood as that classical global-in-time solution of p-system exists if and only if the initial data are nowhere compressive. 

If \eqref{p_lemma_con} is not satisfied at any point, that is, if the initial data contain some compression, then gradient blowup happens in finite time.
\end{remark}

In order to prove Theorem \ref{p_sing_thm}, the following observation plays an important role. From \eqref{p_y_eq} and \eqref{p_q_eq}, using comparison principle for ODEs, with the help of the following two non-negative constants $Y$ and $Q$ defined by 
\begin{equation}
Y=\max\Big\{0,\ \sup_x\{y(x,0)\}\Big\}, \quad  Q= \max\Big\{0,\ \sup_x\{q(x,0)\}\Big\},
\end{equation}
it is easy to see the following lemma holds.  

\begin{lemma}\label{lemma_p_2} If $(\tau_0(x), u_0(x))$ satisfy Assumption \ref{p-assumption}, it holds for $C^1$ solution $(\tau, u)(x,t)$  of (\ref{cp}) that
\[y(x,t)\leq Y, \quad \text{and} \quad  q(x,t)\leq Q\,.\]
\end{lemma}

With the help of this Lemma \ref{lemma_p_2}, we are able to prove the following key estimate on the lower bound of density (equivalently, upper bound of $\tau$), for $1<\gamma<3$, covering most physical cases for 
polytropic gases. 

\begin{lemma}\label{density_low_bound_1-3}  Let $(\tau, u)(x,t)$ be a $C^1$ solution of (\ref{cp}) defined on time interval $[0, T]$ for some $T>0$, with initial data $(\tau_0(x), u_0(x))$ satisfying conditions in Assumption \ref{p-assumption}. If $1<\gamma<3$,  then for any $x\in \mathbb{R} $ and $t\in [0, T)$,  there is a positive constant $K_0$ depending only on $\gamma$, such that 
\[
\tau(x, t)\le \big[{\tau_0}^{\frac{3-\gamma}{4}}(x)+K_0(Y+Q) t\big]^{\frac{4}{3-\gamma}}.
\]
\end{lemma}
\begin{proof} From the definition of $y$ and $q$, it is clear that 
   $$y=s_x\sqrt{\frac{c}{K_c}},\ \quad q=r_x \sqrt{\frac{c}{K_c}} ,$$
which implies that 
$$(y+q)=\sqrt{\frac{c}{K_c}} (r_x+s_x)=2u_x\sqrt{\frac{c}{K_c}}.$$
Therefore, we read from the mass equation that 
$$\sqrt{c}\,\tau_t=\frac12 \sqrt{K_c}(y+q).$$
Using the formula of sound speed \eqref{c def}, Lemma \ref{lemma_p_2}, one finds 
\begin{equation}
\tau^{-\frac{\gamma+1}{4}}\tau_t\le \frac12(K_1\gamma)^{-\frac14}\sqrt{K_c}(Y+Q).
\end{equation}

When $1<\gamma<3$, $\frac{\gamma+1}{4}<1$, then for any $x\in \mathbb{R}$, and $t\in [0, T)$,  a simple time integration shows that 

$$\tau(x, t)\le \big[{\tau_0}^{\frac{3-\gamma}{4}}(x)+K_0(Y+Q) t\big]^{\frac{4}{3-\gamma}},$$
where $K_0=\frac{3-\gamma}{8}(K_1\gamma)^{-\frac14}\sqrt{K_c}$. 
This completes the proof of this lemma.
\end{proof}

\begin{remark} We remark that, 
for purely rarefactive initial data, i.e. the initial data satisfying the conditions of Assumption 2.1 and \eqref{p_lemma_con}, Lin \cite{lin2} proved that the density of  any Lipschitz solution of \eqref{cp} has a 
positive lower bound of order $\frac{1}{1+t}$ through a relatively complicated approximation 
generated by a polygonal scheme. This Lemma \ref{density_low_bound_1-3} works for general data as long 
as $\gamma\in (1, 3)$. Although the time-dependent bound is not as sharp as that in \cite{lin2}, the proof is 
much simpler and elementary. A generalization of \cite{lin2} with $\frac{O(1)}{1+t}$ bound on density to general initial data for all $\gamma>1$
has been carried out in our work \cite{CPZ}. This $\frac{O(1)}{1+t}$ rate is optimal for generic $C^1$ initial data due to the example in \cite{courant, Riemann}. 
\end{remark}

We now give a proof for Theorem \ref{p_sing_thm}.
\begin{proof}[\bf Proof of Theorem \ref{p_sing_thm}]

\noindent\underline{\bf 1) Sufficiency}.  In this part, we prove that under Assumption 2.1, if the initial data satisfy the condition \eqref{p_lemma_con}, then problem \eqref{cp} admits a unique global $C^1$ solution. As a matter of fact, this is a direct consequence of the result presented in \cite{lin2}.  We give an outline here. 

Recall that  the local-in-time existence of $C^1$ solutions for \eqref{cp} can be proved by classical method, c.f. Theorem 4.1 on page 66 of \cite{Li-Yu}, see also \cite{LiBook, Dafermos}, where the life-span of classical solution depends on the $C^1$-norm of the initial data and the positive lower bound of $\tau_0$. The main idea is to use the standard continuity argument to extend the local classical solutions to global with 
a priori estimates in $L^\infty$ and Lipschitz norms of $(\tau, u)(x,t)$.  Indeed, the uniform $L^\infty$ bounds of 
$u(x,t)$ and the lower bound of $\tau(x, t)$ follow from those of $(r, s)(x,t)$ which are constant along their characteristics, respectively; see \eqref{3.1}. For Lipschitz norms, we know from \eqref{p_lemma_con}, that $y(x,0)\ge 0$ and $q(x, 0)\ge 0$, and thus $\|y(x,t)\|_{L^\infty}\le Y$, and $\|q(x,t)\|_{L^\infty}\le Q$. Now, the result of \cite{lin2} offers that, for initial data satisfying Assumption 2.1 and \eqref{p_lemma_con}, for any Lipschitz continuous solutions of \eqref{cp}, there is a positive constant ${\bar K}_0$, independent of time, such that 
 $$\tau\le {\bar K}_0 (1+t),$$ 
which gives the upper bound of $\tau(x, t)$. Furthermore, we deduce from the definitions of $y$, $q$, and $\eta$ that, there exists a function
${\tilde C}(t)$ satisfying $1\le {\tilde C}(t)<\infty$ for any positive finite time $t$ that 
$$\|(r_x, s_x, \tau_x, u_x)(x, t)\|_{L^\infty}\le {\tilde C}(t).$$
Now, if the maximal existence time $T_*$ of $C^1$ solution is finite, then ${\tilde C}({T_*})$ is finite, so are $L^\infty$ norm of $(\tau, u)(x, T_*)$. One can then apply these estimates and the local existence result to 
show there exists a positive time ${\bar t}$ such that the $C^1$ solution can be further extended to the time 
interval $[0, T_*+{\bar t})$ contradicting the fact that $T_*$ is maximal. Therefore, $T_*=\infty$.

\noindent\underline{\bf 2) Necessity}. In this part, we shall prove that under Assumption 2.1, if the initial data fail to satisfy the condition \eqref{p_lemma_con} at one point ${x^*}\in \mathbb{R}$, the $C^1$ solution of \eqref{cp} must blow up in its derivatives in finite time. Without loss of generality, we assume that $s_x(x^*,0)<0$, then $y(x^*, 0)<0$. When $\gamma\ge 3$, this was shown in the last section. Here, 
we only have to consider the case $1<\ga<3$, in which $ {a}_2$ vanishes when density goes to zero.  We denote the forward characteristic passing $(x^*,0)$ as $x^+(t)$. In view of \eqref{p_prop_2_proof}, 
$$
\frac{1}{y(x^+(t), t)}=\frac{1}{y(x^*, 0)}+\int_0^t \, {a_2}(x^+(\sigma), \sigma)\;d\sigma\, .
$$
 To show $y$ blows up in finite time, it is enough to show that 
\[
\int_0^\infty \, {a_2(x^+(t), t)}\;dt=\infty\,,
\]
where the integral is along characteristic $x^+(t)$. We read from Lemma \ref{density_low_bound_1-3}, and 
the definition of $a_2$ that
\[
a_2(x^+(t), t)\geq \frac{\ga +1}{4} K_{\tau}^{-\frac{\ga +1}{4}}\big[M_1^{\frac{3-\gamma}{4}}+K_0(Y+Q)t\big]^{-1}.
\]
Hence,
\[\int_0^\infty \, {a_2(x^+(t), t)}\;dt=\infty\,.\]
Therefore, $y$ and $s_x$ blow up in finite time. The proof of the theorem is completed.
\end{proof}

\begin{remark} The method developed here can be applied to p-system with general pressure laws. Under some mild structural conditions, a similar result to Theorem 2.3 has been established in the section 2.4 of our preprint arXiv:1408.6775v2. 
\end{remark}

\section{Full compressible Euler equations}
In this section, we consider the following Cauchy problem of full compressible Euler equations
\begin{equation}\label{fulleuler0}
\begin{cases}
\tau_t-u_x=0\,,\\
u_t+p_x=0\,,\   p=Ke^{\frac{S}{c_v}}\tau^{-\gamma},\ \gamma>1,\\
S_t=0\,,\\
(\tau, u, S)(x, 0)=(\tau_0, u_0, S_0)(x).
\end{cases}
\end{equation}
Here, we replaced the energy equation with entropy equation. For smooth solutions, we see that $S(x,t)=S_0(x):=S(x)$. Throughout this section, we require that the initial data $(\tau_0, u_0, S_0)(x)$ satisfy conditions in 
the following assumption. 
\begin{assumption}\label{assu} Assume that $(\tau_0(x), u_0(x))\in C^1(\mathbb{R})$, $S_0(x)\in C^2(\mathbb{R})$,  and there are uniform positive constants $M_1$ and $M_2$ such that 
$$\|(\tau_0, u_0)(x)\|_{C^1}+\|S_0(x)\|_{C^2}\le M_1,\ \tau_0\ge M_2.$$
\end{assumption}
For smooth solutions, it is often convenient to choose some new variables. Define 
\beq
   m:=e^{\frac{S}{2c_v}}>0\label{m def}, \quad c:=\sqrt{-p_\tau}=
  \sqrt{K\,\gamma}\,{\tau}^{-\frac{\gamma+1}{2}}\,e^{\frac{S}{2c_v}}\, ,
\eeq
and 
\beq
   \eta := \int^\infty_\tau{\frac{c}{m}\,d\tau}
         = \TS\frac{2\sqrt{K\gamma}}{\gamma-1}\,
\tau^{-\frac{\gamma-1}{2}}>0\,,\label{z def}
\eeq
where $c$ is the nonlinear Lagrangian sound speed. Direct calculations show that (c.f. \cite{G3, G6})
\begin{align}
  \tau&=K_{\tau}\,\eta^{-\frac{2}{\gamma-1}}\,,\nn\\
  p&=K_p\, m^2\, \eta^{\frac{2\gamma}{\gamma-1}}\,,\label{tau p c}\\
  c&=c(\eta,m)=K_c\, m\, \eta^{\frac{\gamma+1}{\gamma-1}}\,.\nn
\end{align}
We remark that, we still use $\eta$, $c$, and many other functions appeared in Section 2 for full Euler equations. These functions are natural extensions from isentropic flows to adiabatic ones, in the sense that they are different to each other only by a positive constant multiple when $S$ is chosen as a constant. 

For $C^1$ solutions, the problem \eqref{fulleuler0} is equivalent to (c.f. \cite{Dafermos,smoller})
\begin{equation}\label{fulleuler1}
\begin{cases}
  \eta_t+\frac{c}{m}\,u_x=0\,, \\
  u_t+m\,c\,\eta_x+2\frac{p}{m}\,m_x=0\,,\\
  m_t=0\,,\\
(\eta, u, m)(x, 0)=(\eta_0, u_0, m_0)(x)=(\eta(\tau_0(x)), u_0(x), m(S_0(x))).
\end{cases}
\end{equation}

Due to the linear degeneracy, in the regime of smooth solutions, $m$ is independent of time, we thus fix  
$m=m(x)=m_0(x)$ in the rest of this paper. Therefore, formally, one can still treat \eqref{fulleuler1} as 
a system of two (significant) equations, with fluxes (pressure) depending on $x$ explicitly. Like in the case of isentropic flows, two truly nonlinear characteristic fields are 
\beq\label{pmc_full}
  \frac{dx^+}{dt}=c \com{and} \frac{dx^-}{dt}=-c\,,
\eeq
and we denote the corresponding directional derivatives along these by
\[
  \pp := \dbyd t+c\,\dbyd x \com{and}
  \pn := \dbyd t-c\,\dbyd x\,,
\]
respectively. Comparing with p-system, one of significant differences for full Euler system is the disappearance of 
Riemann invariances, in fact, the Riemann variables are 
\beq
  r:=u-m\,\eta\,,\qquad s:=u+m\,\eta\,.
\label{r_s_def}
\eeq
which vary along characteristics
\begin{align}
  \pp s&=\frac{1}{2\gamma}\,\frac{c\, m_x}{m}\,(s-r)\,,\label{s_eqn}\\
  \pn r&=\frac{1}{2\gamma}\,\frac{c\, m_x}{m}\,(s-r)\,.\label{r_eqn}
\end{align}
Therefore, without smallness assumption of the solutions, the first non-trivial question one encounters is 
how to achieve $L^\infty$ estimates on the solutions. We remark that this question is trivial for isentropic 
case since $r$ and $s$ are invariant along their characteristics. Fortunately, this question is answered recently by G. Chen, R. Young and Q. Zhang in \cite{G6} under the following additional condition: 
\begin{assumption}\label{BV} Assume that the initial entropy $S_0(x)$
has finite total variation, so that
\beq
  V := \frac{1}{2c_v}\int_{-\infty}^{+\infty}|S'(x)|\;dx
     = \int_{-\infty}^{+\infty}\frac{|m'(x)|}{m(x)}\;dx<\infty\,,
\label{Vdef}
\eeq
\end{assumption}

From Assumption 3.1, it is clear that there are positive constants $M_L$, $M_U$, $M_s$ and $M_r$ such that 
\beq
  0 < M_L < m(x) < M_U\,, \quad |s_0(x)|<M_s,\ \quad 
  |r_0(x)|<M_r\,.
\label{m_bounds}
\eeq
For ${\ol V}=\frac{V}{2\gamma}$,  we now define
\begin{align*}
  N_1 &:= M_s+\ol V\,M_r+\ol V\,(\ol V\,M_s+{\ol V}^2\,M_r)
	\,e^{{\ol V}^2},\\
  N_2 &:= M_r+\ol V\,M_s+\ol V\,(\ol V\,M_r+{\ol V}^2\,M_s)
	\,e^{{\ol V}^2}.
\end{align*}

The following proposition is proved in \cite{G6} by a highly non-trivial characteristic method.

\begin{proposition}{\em \cite{G6}}
\label{Thm_upper} Assume the initial data $(\tau_0, u_0, S_0)(x)$ satisfy the conditions in Assumptions 3.1 and 3.2. If $(\tau(x, t), u(x,t), S(x))$ is a  $C^1$ solution of \eqref{fulleuler0} for $t\in[0,T)$ for some positive $T$, then it holds that 
\beq |s(x,t)|\le N_1{M_U}^{\frac{1}{2\gamma}}, \quad |r(x,t)|\le N_2{M_U}^{\frac{1}{2\gamma}},
\eeq
\beq\label{u_rho_bounds}
   |u(x,t)|\leq\frac{N_1+N_2}{2}{M_U}^{\frac{1}{2\gamma}},\ 
   \eta(x,t)\leq\frac{N_1+N_2}{2}{M_L}^{\frac{1}{2\gamma}-1}:=E_U.
\eeq
Therefore, there are positive constants $M_{\rho}$ such that 
\beq   \rho\le M_{\rho},\quad \tau\ge \frac{1}{M_{\rho}}.\eeq
\end{proposition}

The second major obstacle appears in the equations of gradient variables. Like in p-system, following the wisdoms of many previous works, c.f. \cite{G3, lax2, linliuyang}, a good choice is 
\begin{align}
  y &:= m^{-\frac{3(3-\gamma)}{2(3\gamma-1)}}\,
       \eta^{\frac{\gamma+1}{2(\gamma-1)}}\,
       (s_x - {\TS\frac{2}{3\gamma-1}}\,m_x\,\eta), \nn\\
  q &:= m^{-\frac{3(3-\gamma)}{2(3\gamma-1)}}\,
       \eta^{\frac{\gamma+1}{2(\gamma-1)}}\,
       (r_x + {\TS\frac{2}{3\gamma-1}}\,m_x\,\eta)\,,
\label{intr main}
\end{align}
which satisfy
\begin{align} 
  \partial_+ y &= a_0- a_2 \, y^2, \nn\\
  \partial_- q &= a_0- a_2 \, q^2,
\label{yq odes}
\end{align}
where
\begin{align}
  {a}_0 &:= {\TS\frac{K_c}{\gamma}}\,
        \big[{\TS\frac{\gamma-1}{3\gamma-1}}\,m\,m_{xx}
         - {\TS\frac{(3\gamma+1) (\gamma-1)}{(3\gamma-1)^2}}\,m_x^2\big]\,
	m^{-\frac{3(3-\gamma)}{2(3\gamma-1)}}\,
         \eta^{\frac{3(\gamma+1)}{2(\gamma-1)}+1},\nn\\
  {a}_2 &:= K_c\,{\TS\frac{\gamma+1}{2(\gamma-1)}}\,
	m^{\frac{3(3-\gamma)}{2(3\gamma-1)}}\,
        \eta^{\frac{3-\gamma}{2(\gamma-1)}}.
\label{adefs}
\end{align}
Clearly, $a_0=0$ if $S_0(x)$ (thus $m(x)$) is a constant. For general adiabatic flows, $a_0$ is not constant 
zero. \eqref{yq odes} are not in Riccati type, and these different ODE structures lead to 
different behaviors of solutions, this is more crucial when initial data are not small perturbation around a constant state. 
Indeed, the classical theory of \cite{lizhoukong0,lizhoukong,Liu1} confirms that, when initial data are arbitrarily small near a constant sate away from vacuum, $y$ and/or $q$ blows up in finite time 
if there is some nonlinear compression (under their notations) at some point $x\in\mathbb{R}$, under a so-called nonlinear wave condition \cite{Liu1}. We also note that the choice of gradient variables in \cite{lizhoukong0,lizhoukong,Liu1} is slightly different from our choices here. Our choices of $y$ and $q$ seem better for large solutions. From now on, we adapt the notions to call the initial data are compressive at 
$x$ if  $y(x,0)<0$ or $q(x, 0)<0$, and rarefactive at $x$ if $y(x,0)\ge 0$ or $q(x, 0)\ge 0$. We first present an example to show that weak initial compression does not necessary develop finite time gradient blowup. 

\begin{example}\label{keyex} For any $C^1$ functions $S(x)$ and $\tau(x)>0$, 
\[
u=0,\quad S=S(x)\quad \text{and}\quad \tau=\tau(x)
\]
is a global $C^1$ (stationary) solution of \eqref{fulleuler0} if in the initial data $S$ and $\tau$ are chosen such that
\beq\label{pxx0}
p_x(\tau(x), S(x))=0.
\eeq
Therefore, if we choose a smooth non-constant  function $S(x)$, then choose 
\beq 
        \tau(x)=K_{\tau, S}\ \exp\left\{\frac{S(x)}{\gamma c_v}\right\},
\eeq 
for any positive constant $K_{\tau, S}$, $(\tau(x), 0, S(x))$ is a smooth stationary solution of \eqref{fulleuler0}. In particular, if one chooses $S(x)$ to be a non-constant periodic function, so is $\tau$, this gives a non-constant solution of \eqref{fulleuler0} which is periodic in both space and time.  In order to 
fulfill the condition in Assumption \ref{BV}, a choice of $S(x)$ is $\frac{1}{x^2+1}$.

For such class of solutions  $(\tau(x), 0, S(x))$, a direct calculation shows 
\beq \label{excom}
-q(x)=y(x)=\textstyle\frac{\gamma-1}{\gamma(3\gamma-1)}m_x m^{\frac{3(\gamma-3)}{2(3\gamma-1)}}
\eta^{\frac{3\gamma-1}{2(\gamma-1)}}\,.
\eeq
which is non-zero at point $x$ if $S'(x)\neq 0$. Therefore, either $q(x)<0$ or $y(x)<0$, but no singularity 
forms in the solution. 
\end{example}

\begin{remark} At first glance, it seems that this example contradicts to  Liu's result \cite{Liu1}. We remark that this class of stationary solutions does not satisfy the nonlinear wave condition in \cite{Liu1} for singularity formation, and therefore there is no contradiction. 
\end{remark}

This example shows that weak nonlinear compression in initial data does not necessarily lead 
to finite time singularity formation especially when data could be large, see also some related discussion in \cite{G6}. This motivates our search for 
a critical strength of the nonlinear compression which offers finite time gradient blowup, which will 
be carried out in the next two subsections. In particular, the Section 3.1 is for the case when initial entropy has finite total variation, c.f. Assumption 3.2; while the Section 3.2 contains results without this condition. 
 In addition to these obstacles, like in the case of isentropic flows, we still need to further generalize our 
method in p-system to non-isentropic case to find a sharp enough time-dependent density lower bound.

\subsection{Singularity formation: $\|S_0(x)\|_{BV}<\infty$}

In this subsection, we assume that the initial data $(\tau_0(x), u_0(x), S_0(x))$ satisfy the conditions in 
Assumptions 3.1 and 3.2, so the estimates in Proposition 3.3 hold. 

The structure of \eqref{yq odes} leads us to study the ratio $\frac{a_0}{a_2}$ which dominates behaviors of 
solutions to \eqref{yq odes}. A direct calculation is carried out as follows

\begin{align}\label{a0overa2}
  {\frac{a_0}{a_2}} =
  {{\TS\frac{2(\gamma-1)^2}{\gamma(\gamma+1)(3\gamma-1)}}\,
        \big(m\,m_{xx}-{\TS\frac{3\gamma+1}{3\gamma-1}}\,m_x^2\big)}\,
    \eta^{\frac{3\gamma-1}{(\gamma-1)}}\,m^{-\frac{3(3-\gamma)}{(3\gamma-1)}}.
\end{align}
We define 
\beq b(x)=S_{xx}-\frac{1}{c_v(3\gamma-1)} S_x^2,
\label{bdef}
\eeq
and it is easy to see that 
\beq \label{bm}
m^2b(x)={2c_v}\big(m\,m_{xx}-{\TS\frac{3\gamma+1}{3\gamma-1}}\,m_x^2\big).
\eeq
Therefore, $b(x)$ has the same sign as $a_0$. 
Also, we note from the definition of $m$ that there is a positive constant $M_3$ such that 
\beq\label{m_xx}
|m\,m_{xx}-{\TS\frac{3\gamma+1}{3\gamma-1}}\,m_x^2|\le M_3.
\eeq
If we define a positive constant $N$ by 

\beq
  N :=
  \begin{cases}
  \sqrt{\frac{2(\gamma-1)^2}{\gamma(\gamma+1)(3\gamma-1)}\,M_3}
        \  E_U^{\frac{3\gamma-1}{2(\gamma-1)}}\  
        M_L^{-\frac{3(3-\gamma)}{2(3\gamma-1)}}, & 1<\gamma<3\,,\\
  \sqrt{\frac{2(\gamma-1)^2}{\gamma(\gamma+1)(3\gamma-1)}\,M_3}
        \,E_U^{\frac{3\gamma-1}{2(\gamma-1)}}\,
        M_U^{-\frac{3(3-\gamma)}{2(3\gamma-1)}}, & \gamma\ge 3\,,
  \end{cases}
\label{Ndef}
\eeq   
we see 
\beq
\frac{a_0}{a_2}\le  N^2.
\eeq

\begin{remark} We remark that $a_2$ is positive, while $a_0$ usually changes sign for physical flows. In fact, 
for physical initial conditions of non-constant $S(x)$, 
the case for $b\le 0$ for all $x\in \mathbb{R}$ cannot happen. Actually, from the relation \eqref{bm}, 
one finds that $b\le 0$ is equivalent to $m\,m_{xx}-{\TS\frac{3\gamma+1}{3\gamma-1}}m_x^2\le 0$, which 
is equivalent to 
$$\big(m^{-\frac{2}{3\gamma-1}}\big)_{xx}\ge 0.$$
Therefore, $m^{-\frac{2}{3\gamma-1}}$ is a convex function over $\mathbb{R}$ if $b\le 0$ for all $x\in \mathbb{R}$. This contradicts the fact that $0<M_L\le m(x)\le M_U$. Similar argument shows that the case for $b\ge 0$ for all $x$ cannot happen either. 
\end{remark}

Like Lemma \ref{lemma_p_2}, we are able to find uniform upper bounds for $y$ and $q$. Since it is a little more complicated than Riccati equation, we can simply compare \eqref{yq odes} with the following ones
\beq
\partial_+ {\tilde y}=a_2(N^2-{\tilde y}^2),\ \quad \partial_- {\tilde q}=a_2(N^2-{\tilde q}^2).
\eeq
Therefore, it is easy to see the following lemma. 
\begin{lemma}\label{full_lemma1}
If $(\tau_0, u_0, S_0)(x)$ satisfy conditions in Assumptions 3.1 and 3.2, it holds for $C^1$ solution 
$(\tau(x, t), u(x, t), S(x))$ of \eqref{fulleuler0} that
\[y(x,t)\leq\max\Big\{N,\ \sup_x\{y(x,0)\}\Big\}=:\bar Y\,,\]
\[ q(x,t)\leq\max\Big\{N,\ \sup_x\{q(x,0)\}\Big\}=:\bar Q\,.\]
\end{lemma}

The following lemma contains density lower bound estimate. 

\begin{lemma}\label{density_low_bound_3.1}  Let $(\tau(x,t), u(x,t), S(x))$ be a $C^1$ solution of (\ref{fulleuler0}) defined on time interval $[0, T]$ for some $T>0$, with initial data $(\tau_0(x), u_0(x), S_0(x))$ satisfying conditions in Assumptions 3.1 and 3.2. If $1<\gamma<3$,  then for any $x\in \mathbb{R} $ and $t\in [0, T)$,  there is a positive constant $K_6$ depending only on $\gamma$ and $M_U$, such that 
\[
\tau(x, t)\le \big[{\tau_0}^{\frac{3-\gamma}{4}}(x)+K_6({\bar Y}+{\bar Q}) t\big]^{\frac{4}{3-\gamma}}.
\]
\end{lemma}
\begin{proof} From the mass equation in \eqref{fulleuler0}, \eqref{r_s_def}, and \eqref{intr main}, it is clear 
that 
\beq 
\begin{split}
\tau_t=u_x &=\frac12 (r_x+s_x)\\
      &=m^{\frac{3(3-\gamma)}{2(3\gamma-1)}}\,
       \eta^{-\frac{\gamma+1}{2(\gamma-1)}}(q+y)\\
     &\le M_U^{\frac{3(3-\gamma)}{2(3\gamma-1)}}\,
       \eta^{-\frac{\gamma+1}{2(\gamma-1)}}({\bar Y}+{\bar Q}),
\end{split}
\eeq
where we have used Lemma \ref{full_lemma1}. With the help of \eqref{z def}, we thus have 
\beq
\tau^{-\frac{\gamma+1}{4}}\tau_t\le  M_U^{\frac{3(3-\gamma)}{2(3\gamma-1)}}\,
       (\TS\frac{2\sqrt{K\gamma}}{\gamma-1})^{-\frac{\gamma+1}{2(\gamma-1)}}({\bar Y}+{\bar Q}),
\eeq
which implies that 
\beq
\tau(x,t)\le \big[\tau_0(x)+ K_6 ({\bar Y}+{\bar Q}) t\big]^{\frac{4}{3-\gamma}},
\eeq
where 
$$K_6=\frac{3-\gamma}{4}M_U^{\frac{3(3-\gamma)}{2(3\gamma-1)}}\,
       (\TS\frac{2\sqrt{K\gamma}}{\gamma-1})^{-\frac{\gamma+1}{2(\gamma-1)}}.$$
\end{proof}

In the following theorem, we  show that $N$ is a critical measurement for the strength of initial nonlinear compression, which leads to finite time gradient blowup of solutions.  

\begin{theorem}
\label{Thm singularity2}
For $\gamma>1$, if $(\tau_0(x), u_0(x), S_0(x))$ satisfy conditions in Assumption 3.1 and 3.2, and if 
for $N$ defined in \eqref{Ndef}, it holds that 
\beq
  \inf_x\;\Big\{\  y(x,0),\ q(x,0)\ \Big\} < -N\,,
\label{yq-N}
\eeq
then for the $C^1$ solutions $(\tau(x,t), u(x,t), S(x))$ of \eqref{fulleuler0},   $|u_x|$ and/or $|\tau_x|$ blows up in finite time.
\end{theorem} 

\begin{remark}\label{re_3.5}
The case when $\gamma\geq3$ was carried out  in \cite{G6}, where density lower bound is not needed. 
\cite{G6} also mentioned a similar result of this Theorem with the a priori assumption of a positive constant lower bound on density. See Theorem 2.4 of \cite{G6} for details. 
\end{remark}

\begin{proof} 
Suppose that \eq{yq-N} holds. Without loss of generality,
we can assume that $\inf_x y(x, 0)<-N$, the case when $\inf_x q(x, 0)<-N$ is similar. Then there exist $\eps>0$ and $x_0\in \mathbb{R}$
such that
\beq\label{SS90}
  y(x_0,0) <- (1+\eps)\,N\,.
\eeq
We denote the forward characteristic passing $(x_0,0)$ as $x^+(t)$. Along this characteristic $x^+(t)$, from the definition of $N$, we have for any $t\ge 0$ such that $x^+(t)$ is well-defined,
\[
 \pp y(x^+(t), t) =a_2(\frac{a_0}{a_2}-y^2)<0, \com{and}  y(x^+(t), t)\le y(x_0,0)  < -(1+\eps)\,N.
\]
Therefore, 
$$\frac{y^2(x^+(t), t)}{(1+\eps)^2}> N^2\ge \frac{a_0}{a_2},$$
which implies that 
\[ 
  \pp y (x^+(t), t)= a_2(\frac{a_0}{a_2}-y^2(x^+(t), t))
       <-{\TS\frac{\eps(2+\eps)}{(1+\eps)^2}}\, a_2\,y^2(x^+(t),t)\,.
\]
Integrating it in time, we get 
\beq
  \frac{1}{y(x^+(t), t))} \ge {\frac{1}{y(x_0,0)} + {\TS\frac{\eps(2+\eps)}{(1+\eps)^2}}
    \int_0^t {a_2}(x^+(\sigma), \sigma)\;d\sigma}\,,
\label{SS9 1}
\eeq 
where the integral is along the forward characteristic. To show $y$ blows up in finite time, it is enough to show that 
\beq\label{blowupa2}
\int_0^\infty \, {a_2(x^+(t), t)}\;dt=\infty\,.
\eeq

When $\gamma\geq 3$, from the definition of $a_2$ in \eqref{adefs}, we see that 
$$a_2\ge  K_c\,{\TS\frac{\gamma+1}{2(\gamma-1)}}\,
	M_L^{\frac{3(3-\gamma)}{2(3\gamma-1)}}\,
        E_U^{\frac{3-\gamma}{2(\gamma-1)}},$$
thus \eqref{blowupa2} follows. 

When $1<\gamma<3$, we read from Lemma \ref{density_low_bound_3.1}, and 
the definition of $a_2$ in \eqref{adefs} that
\[
a_2(x^+(t), t)\geq K_c\,{\TS\frac{\gamma+1}{2(\gamma-1)}}\,
	M_L^{\frac{3(3-\gamma)}{2(3\gamma-1)}}  (\TS\frac{2\sqrt{K\gamma}}{\gamma-1})^{\frac{3-\gamma}{2(\gamma-1)}}\big[M_1^{\frac{3-\gamma}{4}}+K_6({\bar Y}+{\bar Q})t\big]^{-1},
\]
therefore, \eqref{blowupa2} also follows. 

Therefore,  for any $\gamma>1$, $y$ and $s_x$ blow up in finite time. The proof of the theorem is completed.

\end{proof}

\subsection{Singularity formation for general entropy function}
We remark that singularity formation in Euler equations is a local behavior. We will remove several global constraints on entropy functions imposed in last subsection to include physically interesting cases such as  spatially periodic solutions. For this purpose, in this subsection,  we only impose conditions in Assumption 3.1 for the initial data, but not Assumption 3.2. 

Without Assumption 3.2, we do not have the global uniform $L^\infty$ estimates in Proposition \ref{Thm_upper}. However, we note that this result is proved by characteristic method, we thus could 
follow the same argument as in \cite{G6} to establish a local version. 

For this purpose, we fix two initial points $\alpha<\beta\in \mathbb{R}$, denote the forward characteristic starting from $(\alpha, 0)$ by $x_\alpha^+(t)$ and the backward characteristic starting from $(\beta, 0)$ by 
$x_\beta^-(t)$. Assume that $(\tau(x,t), u(x,t), S(x))$ is a $C^1$ solution of \eqref{fulleuler0} on the 
time interval $[0, T_1]$ for some positive $T_1$. In the trapezoid showed in figure 2 below, the top edge $t=T\le T_1$ can shrink into one point, if $x_\alpha^+(T)=x_\beta^-(T)$. We denote this trapezoid domain 
by $\Omega_{\alpha, \beta, T}$, which is determined by the initial interval $[\alpha, \beta]$, $x_\alpha^+(t)$, $x_\beta^-(t)$, and $t=T$. 

	\begin{figure}[htp] \centering\label{fig1}
		\includegraphics[width=.5\textwidth]{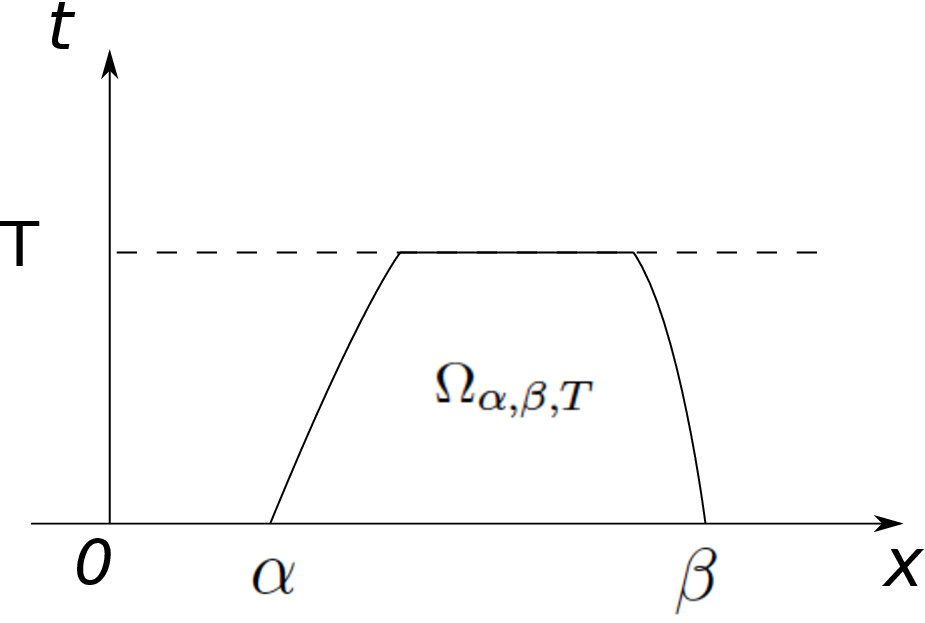}
		\caption{A domain of determination $\Omega_{\alpha,\beta,T}$}
	\end{figure}
With the help of Assumption 3.1, on the interval $[\alpha, \beta]$, one has 
\beq\label{newv}
V_{\alpha,\beta}:= \frac{1}{2c_v}\int_{\alpha}^{\beta}  |S'(\xi)|\;d\xi
     \le \frac{1}{2c_v}M_1|\beta-\alpha|.
\eeq
Therefore, if we define $\bar V_{\alpha,\beta} := \frac{ V_{\alpha,\beta}}{2\gamma}$, and 
\begin{align*}
  {N_1}_{\alpha,\beta} &:= M_s+\bar V_{\alpha,\beta}\,M_r+\bar V_{\alpha,\beta}\,(\bar V_{\alpha,\beta}\,M_s+{\bar V_{\alpha,\beta}}^2\,M_r)
	\,e^{{\bar V_{\alpha,\beta}}^2},\\
 {N_2}_{\alpha,\beta} &:= M_r+\bar V_{\alpha,\beta}\,M_s+\bar V_{\alpha,\beta}\,(\bar V_{\alpha,\beta}\,M_r+{\bar V_{\alpha,\beta}}^2\,M_s)
	\,e^{{\bar V_{\alpha,\beta}}^2},
\end{align*}
the same proof in \cite{G6} gives 

\begin{proposition}
\label{Thm_upperlocal} Assume the initial data $(\tau_0, u_0, S_0)(x)$ satisfy the conditions in Assumption 3.1. If $(\tau(x, t), u(x,t), S_0(x))$ is a  $C^1$ solution of \eqref{fulleuler0} for $t\in[0,T_1)$ for some positive $T_1$, then it holds, for every point $(x,t)\in \Omega_{\alpha,\beta, T}$ and $T\le T_1$, that
\beq\label{u_rho_bounds2}
\begin{split} &|s(x,t)|\le  {N_1}_{\alpha,\beta} {M_U}^{\frac{1}{2\gamma}}, \quad |r(x,t)|\le  {N_2}_{\alpha,\beta}{M_U}^{\frac{1}{2\gamma}}\\
  & |u(x,t)|\leq\frac{ {N_1}_{\alpha,\beta} + {N_2}_{\alpha,\beta} }{2}{M_U}^{\frac{1}{2\gamma}},\\
   &\eta(x,t)\leq\frac{ {N_1}_{\alpha,\beta} + {N_2}_{\alpha,\beta} }{2}{M_L}^{\frac{1}{2\gamma}-1}:={\tilde E}_U.
\end{split}
\eeq
Therefore, there are positive constants ${\tilde M}_{\rho}$ such that 
\beq   \rho\le {\tilde M}_{\rho},\quad \tau\ge \frac{1}{{\tilde M}_{\rho}}.\eeq
\end{proposition}

For later use, we give an estimate on the expected time $T_{\alpha,\beta}$ where $x_\alpha^+(t)$ and $x_\beta^-(t)$ intersect if no singularity develops before this time. Using \eqref{pmc_full}, a simple calculation shows that 
\[T_{\alpha,\beta}\ge \frac{\beta-\alpha}{2}M_U^{-1}{\tilde E}_U^{-\frac{\gamma+1}{\gamma-1}}.
\] 

Using the definition of $M_3$ in \eqref{m_xx}, we define 
\beq
 N_{\alpha,\beta}:=
  \begin{cases}
  \sqrt{\frac{2(\gamma-1)^2}{\gamma(\gamma+1)(3\gamma-1)}\,M_3}
        \  {{\tilde E}_U}^{\frac{3\gamma-1}{2(\gamma-1)}}\  
        M_L^{-\frac{3(3-\gamma)}{2(3\gamma-1)}}, & 1<\gamma<3\,,\\
  \sqrt{\frac{2(\gamma-1)^2}{\gamma(\gamma+1)(3\gamma-1)}\,M_3}
        \,{{\tilde E}_U}^{\frac{3\gamma-1}{2(\gamma-1)}}\,
        M_U^{-\frac{3(3-\gamma)}{2(3\gamma-1)}}, & \gamma\ge 3\ .
  \end{cases}
\label{Nxzdef}
\eeq 
Therefore, 
\beq
\frac{a_0}{a_2}(x, t)\le  { N}_{\alpha, \beta}^2,\  \forall (x, t)\in\Omega_{\alpha, \beta, T}.
\eeq

Furthermore, we define 
\beq
\begin{split} &{\tilde Y}:=\max\Big\{{ N}_{\alpha, \beta},\ \sup_{x\in[\alpha, \beta]} \{y(x,0)\}\Big\},\\
&{\tilde Q}:=\max\Big\{{N}_{\alpha, \beta},\ \sup_{x\in[\alpha, \beta]}\{q(x,0)\}\Big\}.
\end{split}\eeq
It is now clear that, the same method used in the proof of Lemma \ref{density_low_bound_3.1} gives 

\begin{lemma}\label{density_low_bound_3.2}  Let $(\tau(x,t), u(x,t), S(x))$ be a $C^1$ solution of (\ref{fulleuler0}) defined in $\Omega_{\alpha, \beta, T}$, with initial data $(\tau_0(x), u_0(x), S_0(x))$ satisfying conditions in Assumption 3.1. If $1<\gamma<3$,  then for any $(x, t)\in \Omega_{\alpha, \beta, T}$, there is a positive constant ${\tilde K}_6$ depending only on $\gamma$ and $M_U$, such that 
\[
\tau(x, t)\le \big[{\tau_0}^{\frac{3-\gamma}{4}}(x)+{\tilde K}_6({\tilde Y}+{\tilde Q}) t\big]^{\frac{4}{3-\gamma}}.
\]
\end{lemma}
Therefore, we have the following estimate on $a_2$ for any $(x, t)\in \Omega_{\alpha, \beta, T}$
\beq\label{a2new}
a_2(x,t)\ge \begin{cases}
K_c\,{\TS\frac{\gamma+1}{2(\gamma-1)}}\,
	M_L^{\frac{3(3-\gamma)}{2(3\gamma-1)}}\,
        {\tilde E}_U^{\frac{3-\gamma}{2(\gamma-1)}}:=K_8,\quad \text{if}\ \gamma\ge 3,\\
K_7\big[M_2^{\frac{3-\gamma}{4}}+{\tilde K}_6({\tilde Y}+{\tilde Q})t\big]^{-1}, \ \text{if}\ 1<\gamma<3.
\end{cases}
\eeq
where 
$$K_7=K_c\,{\TS\frac{\gamma+1}{2(\gamma-1)}}\,
	M_L^{\frac{3(3-\gamma)}{2(3\gamma-1)}}  (\TS\frac{2\sqrt{K\gamma}}{\gamma-1})^{\frac{3-\gamma}{2(\gamma-1)}}.$$
We further introduce the following constants $K_9$ and $K_{10}$ by 
\beq \label{final constant}
K_{9}={\tilde K}_6({\tilde Y}+{\tilde Q})M_2^{\frac{\gamma-3}{4}}, \quad K_{10}=K_7 M_2^{\frac{\gamma-3}{4}},
\eeq
so that 
\beq \label{a213}
a_2\ge K_{10}[1+K_9t]^{-1}, \ \text{if}\ 1<\gamma<3.
\eeq
We introduce another below constant to assist the measurement on the nonlinear compression. Let positive 
constant $B_{\alpha,\beta}$ be a solution of 
\beq\label{thm3_con}
\frac{B_{\alpha,\beta}(2+B_{\alpha,\beta})}{{(1+B_{\alpha,\beta})}}\ge
\begin{cases}
\left(K_8\,{N_{\alpha,\beta}}{T_{\alpha,\beta}}\right)^{-1},\quad\quad\  \ \gamma\geq3\,,\vspace{.2cm}\\
       \left(\frac{K_{10}}{K_9}\,N_{\alpha,\beta}\,\ln(1+{K_9} T_{\alpha,\beta}) \right )^{-1},\ 
        1<\gamma<3\,.
\end{cases}        
\eeq

\begin{theorem}
\label{Thm singularity3}
Assume the initial data $(\tau_0, u_0, S_0)(x)$ satisfy conditions in Assumption 3.1. If there exists some interval $(\alpha,\beta)$ such that  the initial data satisfy
\beq\label{thm3_1}
\inf_{x\in [\alpha, \beta]} \{y(x,0), q(x, 0)\} < -{ N}_{\alpha,\beta}(1+B_{\alpha,\beta})\,,
\eeq
then  $|u_x|$ and/or $|\tau_x|$ blow up in finite time. 
\end{theorem} 
\begin{remark}
The right hand side of \eqref{thm3_con} only depends on the initial data.
For any given entropy function satisfying conditions in Assumption 3.1, condition \eqref{thm3_con} will be satisfied when $B_{\alpha,\beta}$
is large enough, i.e. $y(x,0)$ or $q(x,0)$ is negative enough. This means that singularity forms in finite time when the initial compression is strong enough somewhere. 

One good choice of $B_{\alpha,\beta}$  is 
\beq\label{thm3_con2}
B_{\alpha,\beta} =\begin{cases}
\left(K_8\,{N_{\alpha,\beta}}{T_{\alpha,\beta}}\right)^{-1},\ \gamma\geq3\,,\vspace{.2cm}\\
       \left(\frac{K_{10}}{K_9}\,N_{\alpha,\beta}\,\ln(1+{K_9} T_{\alpha,\beta}) \right )^{-1},\ 
        1<\gamma<3\,,
\end{cases}        
\eeq

This result is consistent with Theorem \ref{Thm singularity2}. In fact, when the initial entropy has finite total variation, $T_{x,\infty}=\infty$ while $N_{x,\infty}$, ${K_8}$ and ${K_9}$, $K_{10} $ are all finite, so $B_{x,\infty}$ can be arbitrarily small. Hence, if $y(x,0)< -N_{x,\infty}$ or $q(x, 0)<-N_{x, \infty}$ for some $x$, then blowup happens in finite time.
\end{remark}
\begin{proof} We only consider the solution in $\Omega_{\alpha,\beta, T_{\alpha, \beta}}$, and prove
that singularity formation happens in this region. Without loss of generality, we assume that 
there is a point $x_*\in [\alpha, \beta]$ such that $y(x_*, 0)<-{ N}_{\alpha, \beta} (1+B_{\alpha, \beta})$, 
the case for $q$ is similar. Denote the forward characteristic starting from $(x_*, 0)$ by $x^+(t)$. We will show that
$y$ goes to negative infinity along $x^+(t)$ before time $T_{\alpha,\beta}$.  

From \eqref{yq odes}, and the definition of ${ N}_{\alpha,\beta}$ in \eqref{Nxzdef}, it is clear that, along 
$x^+(t)$, for $t\in[0, T_{\alpha, \beta}]$ and as long as solution is $C^1$, it holds that 
\[
\pp y(x^+(t), t) =a_2(\frac{a_0}{a_2}-y^2)<0,\com{and} y(x^+(t), t)< -{ N}_{\alpha,\beta}(1+B_{\alpha,\beta}).
\]

Therefore, 
$$\frac{y^2(x^+(t), t)}{(1+B_{\alpha, \beta})^2}> {N}_{\alpha, \beta}^2\ge \frac{a_0}{a_2},$$
which implies that 
\[ 
  \pp y (x^+(t), t)= a_2(\frac{a_0}{a_2}-y^2(x^+(t), t))
       <-{\TS\frac{B_{\alpha, \beta}(2+B_{\alpha, \beta})}{(1+B_{\alpha,\beta})^2}}\, a_2\,y^2(x^+(t), t)\,.
\]
Integrating it in time, we get 
\beq\label{thm3_2}
  \frac{1}{y(x^+(t), t)} \ge {\frac{1}{y(x_*,0)} +
    \frac{B_{\alpha,\beta}(2+B_{\alpha,\beta})}{(1+B_{\alpha,\beta})^2}\int_0^t\, {a_2}(x^+(\sigma), \sigma)\;d\sigma}\,.
\eeq
where the integral is along the forward characteristic. 

Hence the blowup happens at a time $t_1$ when the right hand side of (\ref{thm3_2})
equals to zero, i.e. when
\beq
-\frac{1}{y(x_*, 0)}=\frac{B_{\alpha,\beta}(2+B_{\alpha,\beta})}{{(1+B_{\alpha,\beta})^2}}\int_0^{t_1}\, {a_2}(x^+(\sigma), \sigma)\;d\sigma\,.
\eeq
It is clear from the estimates on $a_2$ in \eqref{a2new} that such a finite $t_1$ exists. However, we still 
need to show that $t_1<T_{\alpha,\beta}$. From \eqref{thm3_1}, we only need to show that 
\beq\label{thm3_3}
\frac{1}{N_{\alpha,\beta}}\leq\frac{B_{\alpha,\beta}(2+B_{\alpha,\beta})}{{(1+B_{\alpha,\beta})}}\int_0^{T_{\alpha,\beta}}\, {a_2(x^+(t), t)}\;dt\,.
\eeq

When  $\gamma\geq 3$, we read from \eqref{a2new} that $a_2\ge K_8$, \eqref{thm3_3} follows directly from 
\eqref{thm3_con}.

When  $1<\gamma< 3$, we read from \eqref{a213} that 
$$a_2\ge K_{10}[1+{K}_9 t]^{-1},$$
therefore 
\beq\begin{split}
&\frac{B_{\alpha,\beta}(2+B_{\alpha,\beta})}{{(1+B_{\alpha,\beta})}}\int_0^{T_{\alpha,\beta}}\, {a_2(x^+(t), t)}\;dt\\
&\ge \frac{B_{\alpha,\beta}(2+B_{\alpha,\beta})}{{(1+B_{\alpha,\beta})}}\int_0^{T_{\alpha,\beta}}  K_{10}[1+K_9t]^{-1}\ dt\\
&=\frac{B_{\alpha,\beta}(2+B_{\alpha,\beta})}{{(1+B_{\alpha,\beta})}}\frac{K_{10}}{K_9}\ln (1+K_9T_{\alpha,\beta} ),
\end{split}
\eeq
which, together with \eqref{thm3_con}, implies \eqref{thm3_3}.
Hence we complete the proof of this theorem.

\end{proof}

\subsection{Further discussion\label{section_3.5}}

In Subsections 3.1-3.2, we showed that if the initial compression is strong enough, singularity develops in finite time for solutions of \eqref{fulleuler0}. It is also evident by the Example \ref{keyex} that relatively strong compression is necessary to guaranty finite time blowup occurs. One of the questions would be, say in Theorem \ref{Thm singularity2}, does the constant $N$ measures the critical strength efficiently? We now show 
this $N$ is the arguably  best possible one. 

We now revisit the stationary solutions 
\beq\label{ss}(\tau(x), u(x), S(x))=(K_{\tau, S}\ \exp\left\{\frac{S(x)}{\gamma c_v}\right\}, 0, S(x)),
\eeq
for any smooth function $S(x)$ satisfying Assumptions 3.1-3.2, and any positive constant $K_{\tau, S}$, constructed in Example \ref{keyex} . We also recall \eqref{excom} 
\beq\label{final_remark1}
-q(x)=y(x)=\textstyle\frac{\gamma-1}{\gamma(3\gamma-1)}m_x m^{\frac{3(\gamma-3)}{2(3\gamma-1)}}
\eta^{\frac{3\gamma-1}{2(\gamma-1)}}\,.
\eeq
We remark that now everything is fixed except the choice of $S(x)$. $m$, $\eta$ are both fucntions of $S(x)$. 
For convenience, we will use $m(x)=e^{\frac{S}{2c_v}}$ for our argument below.

Note  from \eqref{a0overa2} that if 
\beq\label{a0>0} m\,m_{xx}-{\TS\frac{3\gamma+1}{3\gamma-1}}\,m_x^2\geq 0,
\eeq
$N$ is the best possible upper bound of
\beq\label{final_remark2}
  \TS\sqrt{\frac{a_0}{a_2}} =
  \sqrt{{\frac{2(\gamma-1)^2}{\gamma(\gamma+1)(3\gamma-1)}}\,
        \big(m\,m_{xx}-{\TS\frac{3\gamma+1}{3\gamma-1}}\,m_x^2\big)}\,
    \eta^{\frac{3\gamma-1}{2(\gamma-1)}}\,m^{-\frac{3(3-\gamma)}{2(3\gamma-1)}}\,.
\eeq
We now show that there exists some $S(x)$ (or equivalently $m(x)$) so that initially $|y(x)|=|q(x)|=\TS\sqrt{\frac{a_0}{a_2}}$ at 
some $x$. Comparing \eqref{final_remark1} with \eqref{final_remark2}, we see this happens when
\[\textstyle(\frac{\gamma-1}{\gamma(3\gamma-1)})^2m_x^2={\TS\frac{2(\gamma-1)^2}{\gamma(\gamma+1)(3\gamma-1)}}\,
        \big(m\,m_{xx}-{\TS\frac{3\gamma+1}{3\gamma-1}}\,m_x^2\big),\]
which is equivalent to
\beq\label{final_remark3}
mm_{xx}-{\TS\frac{3\gamma+1}{3\gamma-1}}\,m_x^2-{\TS\frac{\gamma+1}{2\gamma(3\gamma-1)}} m_x^2 =0.
\eeq
It is clear that if $m(x)$ satisfies \eqref{final_remark3}, it satisfies \eqref{a0>0}. A direct calculation shows that 
for positive $m$ (or equivalently a bounded $S(x)$), \eqref{final_remark3} is equivalent to 
\beq\label{final_remark4}
(m^\theta)_{xx}=0,\ \text{for}\quad \theta=1-\frac{6\gamma^2+3\gamma+1}{2\gamma(3\gamma-1)}.
\eeq
Clearly, for any point ${\bar x}\in \mathbb{R}$, we are able to choose a smooth function $m(x)$ such that $m^\theta(x)$ reaches its inflection point at ${\bar x}$ and $m'({\bar x})\neq 0$. Indeed, using the formula 
$\tau=K_{\tau, S}\exp\left\{\frac{S}{\gamma c_v}\right\}$ along with \eqref{final_remark1}, 
$$-q(x)=y(x)=K_{\tau, S}^{-\frac{3\gamma-1}{4}}\frac{\gamma-1}{\theta \gamma(3\gamma-1)}(m^{\theta})_x.$$
Thus it confirms that ${\bar x}$ is exactly the (local) maximum of $|q(x)|=|y(x)|$, which can be easily chosen 
as the global maximum for a class of $m(x)$. 

From the analysis above, it is clear that the constant $N$ is almost an optimal measurement on the strength of 
compression for finite time singularity formation in general, because, if the condition \eqref{yq-N} fails in Theorem \ref{Thm singularity2}, then there exists a class of initial data admitting global stationary solutions of the form \eqref{ss}, such that 
\beq
  \inf_x\;\Big\{\  y(x,0),\ q(x,0)\ \Big\} = -N\, .
\eeq
In these examples, the maximum strength of 
compression $N$ is attained. 

\appendix

\section{p-system with general pressure law\label{general_p}}
In this subsection, we generalize the method developed in previous section to the following Cauchy problem for p-system,

\begin{equation}\label{gp}
\begin{cases}
&\tau_t-u_x=0\, ,\\
&u_t+p_x=0\, , \\
&\tau(x,0)=\tau_0(x),\ u(x, 0)=u_0(x),
\end{cases}
\end{equation}
with general pressure law $p(\tau)\in C^3(0,\infty)$ satisfying
\beq\label{p_gen}
p_\tau<0,\ \  p_{\tau\tau}>0
\eeq
and
\beq\label{p_gen2}
 \lim_{\tau\rightarrow 0}p(\tau)=\infty,\ \
\lim_{\tau\rightarrow\infty}p(\tau)=0 
\com{and} \int_1^\infty \sqrt{-p_\tau}~d\tau<\infty.
\eeq
Here condition \eqref{p_gen} is dictated by physics when one uses p-system to model gas dynamics, c.f. \cite{menikoff}. Furthermore, we assume
that 
\beq\label{c_con_gen}
\int_0^1 \sqrt{-p_\tau}~d\tau=\infty 
\eeq
which includes the $\gamma$-law pressure case. We also identified the following condition: 

\begin{assumption} There exists some positive constant $A$, such that for any $\tau>0$,
\beq\label{condition1}
(5+A)(p_{\tau\tau})^2-4p_\tau p_{\tau\tau\tau}\geq 0\,.
\eeq
\end{assumption} 

\begin{remark} This Condition \eqref{condition1} is fairly mild because the constant $A$ can be arbitrarily large. 
For example, the $\gamma$-law pressure $p=k\tau^{-\gamma}$ with $\gamma>0$ satisfies conditions \eqref{condition1} and \eqref{p_gen}, and the pressure $p=k\tau^{-\gamma}$ with $\gamma>1$ satisfies conditions \eqref{condition1}
and \eqref{p_gen}$\sim$\eqref{c_con_gen}.
\end{remark}

Applying Lax's method in Sections 2.1-2.2 to this case (the detailed calculations can be found in \cite{G4}), it is not hard to find the Lagrangian sound speed is 
\[
c\equiv c(\tau)=\sqrt{-p_\tau},
\]
and Riemann invariants
\[
s:=u+\displaystyle\int_\tau^{1} c(\tau) d\tau \com{and} r:=u-\displaystyle\int_\tau^1 c(\tau) d\tau,
\]
which satisfy 
\beq\label{sr_con_gen}
\partial_+ s=0\com{and}  \partial_-r=0\,, \quad \partial_{\pm}=\partial_t\pm c \partial_x.
\eeq
If we define 
\beq
y:=\sqrt{c}\,s_x,\qquad q:=\sqrt{c}\,r_x,\label{GE yq def}
\eeq
then 
\begin{eqnarray} \partial_+y&=&-a({\tau})
y^2, \label{new ode1}\\ \partial_-q&=&-a({\tau}) q^2,
\label{new ode2}
\end{eqnarray} 
where \begin{eqnarray}
a({\tau}):=\frac{p_{{\tau}{\tau}}}{4(-p_{\tau})^{\frac{5}{4}}}>0\label{GE a2}.
\end{eqnarray}

Similar to  Lemma \ref{lemma_p_2}, define non-negative constants
$$Y=\max\Big\{0,\ \sup_x\{y(x,0)\}\Big\}, \quad Q= \max\Big\{0,\ \sup_x\{q(x,0)\}\Big\},$$
we have 
{\begin{lemma}\label{lemma_p_2_gen} If $(\tau_0(x), u_0(x))$ satisfy Assumption \ref{p-assumption}, then it holds for $C^1$ solution $(\tau, u)(x,t)$  of (\ref{gp}) that
\[y(x,t)\leq Y, \quad \text{and} \quad  q(x,t)\leq Q\,.\]
\end{lemma}}

Then we could state our theorem for the general pressure law case. 

\begin{theorem}\label{p_sing_gen}
Assume that initial data  $(\tau_0(x), u_0(x))$ satisfy Assumption 2.1.  The pressure satisfies \eqref{p_gen}$\sim$\eqref{c_con_gen} and the Assumption 2.9. 
Then global-in-time classical solution of (\ref{gp}) exists if and only if
\beq\label{p_lemma_con_gen}
 s_x(x,0)\geq0 \com{and}  r_x(x,0)\geq 0,\com{for all} x\in \mathbb{R}\,.
\eeq
\end{theorem}
\begin{remark}
It is clear from our proof below that for 
the singularity formation, conditions \eqref{p_gen2}$\sim$\eqref{c_con_gen} are not necessary.
\end{remark}
\begin{proof} As usual, one reads from \eqref{sr_con_gen} that $\|(r, s)(x,t)\|_{L^{\infty}}\le \|(r, s)(x, 0)\|_{L^\infty}$. Therefore, on finds uniform $L^\infty$ bounds for $u(x,t)$ and $\displaystyle\int_\tau^{1} c(\tau) d\tau$. It then follows from  \eqref{c_con_gen} that there are positive constant $\tau_{min}$ and 
$c_{max}$ depending only on the initial data such that  
\[
\tau(x, t)\ge  \tau_{min},\ \quad c(x,t)\le c_{max}.
\]

If condition \eqref{p_lemma_con_gen} holds, the global existence could 
be proved in an exactly same way as in the first part of the proof of Theorem \ref{p_sing_thm} together with the positive lower bound on density provided in \cite{lin2}.

If condition \eqref{p_lemma_con_gen} fails, by a similar argument as in the second part of the proof of Theorem \ref{p_sing_thm}, in order to prove singularity formation in finite time, it is sufficient to show
\beq\label{int_a_infty_gen}
\int_0^\infty a({\tau}(x(t),t)~dt=\infty\, ,
\eeq
which is true if we can prove 
\beq\label{a_reci}
\frac{1}{a(\tau(x,t))}= \frac{4(-p_{\tau})^{\frac{5}{4}}}{p_{{\tau}{\tau}}}\leq K_2+K_3 t
\eeq
for some positive constants $K_2$ and $K_3$. Indeed,  a direct computation gives
\[
\frac{1}{2}(y+q)=\frac{1}{2}\sqrt{c}(s_x+r_x)=\sqrt{c}\,u_x=\sqrt{c}\,\tau_t.
\]
Then by Lemma \ref{lemma_p_2_gen}, we have 
\[
\Big(\int_{\tau_{min}}^{\tau}(-p_\tau(\tau))^{\frac{1}{4}}~d\tau\Big)_t=\Big(\int_{\tau_{min}}^{\tau}\sqrt{c(\tau)}~d\tau\Big)_t=\frac{1}{2}(y+q)\leq\frac{1}{2}(Y+Q).
\]
Hence
\beq\label{gen_final}
\int_{\tau_{min}}^{\tau(x,t)}(-p_\tau(\tau))^{\frac{1}{4}}~d\tau\leq \int_{\tau_{min}}^{\tau(x,0)}(-p_\tau(\tau))^{\frac{1}{4}}~d\tau +\frac{1}{2}(Y+Q)t\leq K_4+K_5 t
\eeq
for some positive constants $K_4$ and $K_5$.

Using the fact $\tau>\tau_{min}>0$ and \eqref{gen_final},  \eqref{a_reci} follows if we can show that 
\beq\label{gen_final2}
\Big(\frac{4(-p_{\tau})^{\frac{5}{4}}}{p_{{\tau}{\tau}}}\Big)_\tau\leq A(-p_\tau(\tau))^{\frac{1}{4}},
\eeq
for some positive constant $A$. Since \eqref{gen_final2} follows from \eqref{condition1} in Assumption 2.9, we finish the proof of this theorem. 
\end{proof}

\section*{Acknowledgement}
{ We sincerely appreciate Professor Alberto Bressan  for his very helpful suggestions and discussions when we wrote this paper. We also appreciate the reviewers' helpful comments. The research of R. Pan was supported in part by NSF under grant DMS-1108994. The research of S. Zhu was partially supported 
by National Natural Science Foundation of China under grant 11231006, Natural Science Foundation of Shanghai under grant 14ZR1423100 and  China Scholarship Council.}

\end{document}